\documentclass[runningheads,a4paper]{llncs}
\usepackage{amsmath, amscd, amsfonts, amssymb, graphicx, color}
\usepackage{url}

\urldef{\mailsb}\path| oguz@metu.edu.tr|
\urldef{\mailsa}\path| damla.acar@hacettepe.edu.tr|

\newcommand{\keywords}[1]{\par\addvspace\baselineskip
\noindent\keywordname\enspace\ignorespaces#1}
\usepackage{bbold}

\RequirePackage[T1]{fontenc}
\RequirePackage[utf8]{inputenc}

\newcommand{\BH}{{\rm BH}}
\newcommand{\GH}{{\rm GH}}

\begin{document}


\renewenvironment{proof}{\noindent\textit{Proof.}}{\hfill{$\Box$}}


\newcommand{\binomial}[2]{\left(\begin{array}{c}#1\\#2\end{array}\right)}
\newcommand{\zar}{{\rm zar}}
\newcommand{\an}{{\rm an}}
\newcommand{\red}{{\rm red}}
\newcommand{\codim}{{\rm codim}}
\newcommand{\rank}{{\rm rank}}
\newcommand{\Pic}{{\rm Pic}}
\newcommand{\Div}{{\rm Div}}
\newcommand{\Hom}{{\rm Hom}}
\newcommand{\im}{{\rm im}}
\newcommand{\Spec}{{\rm Spec}}
\newcommand{\sing}{{\rm sing}}
\newcommand{\reg}{{\rm reg}}
\newcommand{\Char}{{\rm char}}
\newcommand{\Tr}{{\rm Tr}}
\newcommand{\tr}{{\rm tr}}
\newcommand{\supp}{{\rm supp}}
\newcommand{\Gal}{{\rm Gal}}
\newcommand{\Min}{{\rm Min \ }}
\newcommand{\Max}{{\rm Max \ }}
\newcommand{\Span}{{\rm Span  }}

\newcommand{\Frob}{{\rm Frob}}
\newcommand{\lcm}{{\rm lcm}}

\newcommand{\ifc}{{\rm if \ }}

\newcommand{\soplus}[1]{\stackrel{#1}{\oplus}}
\newcommand{\dlog}{{\rm dlog}\,}
\newcommand{\limdir}[1]{{\displaystyle{\mathop{\rm
lim}_{\buildrel\longrightarrow\over{#1}}}}\,}
\newcommand{\liminv}[1]{{\displaystyle{\mathop{\rm
lim}_{\buildrel\longleftarrow\over{#1}}}}\,}
\newcommand{\boxtensor}{{\Box\kern-9.03pt\raise1.42pt\hbox{$\times$}}}
\newcommand{\sext}{\mbox{${\mathcal E}xt\,$}}
\newcommand{\shom}{\mbox{${\mathcal H}om\,$}}
\newcommand{\coker}{{\rm coker}\,}
\renewcommand{\iff}{\mbox{ $\Longleftrightarrow$ }}
\newcommand{\onto}{\mbox{$\,\>>>\hspace{-.5cm}\to\hspace{.15cm}$}}

\newcommand{\ord}{\mathrm{ord}}


\newcommand{\sA}{{\mathcal A}}
\newcommand{\sB}{{\mathcal B}}
\newcommand{\sC}{{\mathcal C}}
\newcommand{\sD}{{\mathcal D}}
\newcommand{\sE}{{\mathcal E}}
\newcommand{\sF}{{\mathcal F}}
\newcommand{\sG}{{\mathcal G}}
\newcommand{\sH}{{\mathcal H}}
\newcommand{\sI}{{\mathcal I}}
\newcommand{\sJ}{{\mathcal J}}
\newcommand{\sK}{{\mathcal K}}
\newcommand{\sL}{{\mathcal L}}
\newcommand{\sM}{{\mathcal M}}
\newcommand{\sN}{{\mathcal N}}
\newcommand{\sO}{{\mathcal O}}
\newcommand{\sP}{{\mathcal P}}
\newcommand{\sQ}{{\mathcal Q}}
\newcommand{\sR}{{\mathcal R}}
\newcommand{\sS}{{\mathcal S}}
\newcommand{\sT}{{\mathcal T}}
\newcommand{\sU}{{\mathcal U}}
\newcommand{\sV}{{\mathcal V}}
\newcommand{\sW}{{\mathcal W}}
\newcommand{\sX}{{\mathcal X}}
\newcommand{\sY}{{\mathcal Y}}
\newcommand{\sZ}{{\mathcal Z}}


\newcommand{\A}{{\mathbb A}}
\newcommand{\B}{{\mathbb B}}
\newcommand{\C}{{\mathbb C}}
\newcommand{\D}{{\mathbb D}}
\newcommand{\E}{{\mathbb E}}
\newcommand{\F}{{\mathbb{F}}}
\newcommand{\G}{{\mathbb G}}
\newcommand{\HH}{{\mathbb H}}
\newcommand{\I}{{\mathbb I}}
\newcommand{\J}{{\mathbb J}}
\newcommand{\M}{{\mathbb M}}
\newcommand{\N}{{\mathbb N}}
\renewcommand{\P}{{\mathbb P}}
\newcommand{\Q}{{\mathbb Q}}
\newcommand{\R}{{\mathbb R}}
\newcommand{\T}{{\mathbb T}}
\newcommand{\U}{{\mathbb U}}
\newcommand{\V}{{\mathbb V}}
\newcommand{\W}{{\mathbb W}}
\newcommand{\X}{{\mathbb X}}
\newcommand{\Y}{{\mathbb Y}}
\newcommand{\Z}{{\mathbb Z}}


\newcommand{\be}{\begin{eqnarray}}
\newcommand{\ee}{\end{eqnarray}}
\newcommand{\nn}{{\nonumber}}
\newcommand{\dd}{\displaystyle}
\newcommand{\ra}{\rightarrow}
\newcommand{\bigmid}[1][12]{\mathrel{\left| \rule{0pt}{#1pt}\right.}}
\newcommand{\cl}{${\rm \ell}$}
\newcommand{\clp}{${\rm \ell^\prime}$}

\newcommand{\TODO}[1]
{\par\fbox{\begin{minipage}{0.9\linewidth}\textbf{TODO:} #1\end{minipage}}\par}

\providecommand{\amssub}[1]
{
	\small	
	\textbf{{2020 AMS Subject Classification:}} #1
}


\mainmatter

\title{Butson-Hadamard matrices and Plotkin-optimal $p^k$-ary codes}

\titlerunning{Butson-Hadamard Matrices and codes}

\author{Damla Acar$^1$ \and Bülent Saraç$^1$ \and O\u guz Yayla$^2$}
\authorrunning{Acar, Saraç, Yayla}


	\institute{
		$^1$ Hacettepe University, Department of Mathematics,\\Beytepe, 06800, Ankara, Turkey \\
		\email{damla.acar@hacettepe.edu.tr, bsarac@hacettepe.edu.tr}\\   $^2$    Middle East Technical University, Institute of Applied Mathematics,\\  06800, Ankara, Turkey \\ \email{oguz@metu.edu.tr}   
	}

	\maketitle

\begin{abstract}

A  Butson-Hadamard matrix $H$ is a square matrix of dimension $n$ whose entries are complex roots of unity such that   $HH^{*}= nI $.  In the first part of this work, some new results on generalized Gray map  are studied. In the second part, codes obtained from Butson-Hadamard matrices  and some bounds on the minimum distance of these codes are proved. In particular, we show that the code obtained from a Butson-Hadamard matrix meets the Plotkin bound under a non-homogeneous weight. We also give the parameters of some code families which are obtained from modified Butson-Hadamard matrices under a (non)homogeneous Gray map.
\\
\amssub{05B20,94B60,94B65}
\end{abstract}

\section{Introduction}
A \textit{Hadamard matrix} of order $n$ is an $n\times n $ matrix whose entries are ${1,-1}$ satisfying $HH^{T}= nI_{n}$. Here $I_{n}$ is the $n\times n$ identity matrix.  Hadamard \cite{hadamard1893} conjectured that such matrices  exist for every $n$ that is a multiple of $4$. See \cite{hedayat1978hadamard,horadam2007hadamard,seberry2017orthogonal}   for more information about Hadamard matrices and their applications. Hadamard matrices can be generalized in many ways. Two of them are  Butson-Hadamard (BH) matrices which are introduced by Butson in \cite{butson1962generalized} and generalized Hadamard (GH) matrices by Drake in \cite{drake1979}.
 
A \textit{Butson-Hadamard matrix} $H$ is  an $n \times n$ matrix whose entries are complex roots of unity such that   
\begin{equation*}
HH^{*}=nI_n
\end{equation*}
where $H^{*}$ is the Hermitian transpose of $H$.  If all the entries of $H$ are $k$-th roots of unity then we say that $H$ is a $\BH(n,k)$.

Another generalization of Hadamard matrices is\textit{ (group) generalized
Hadamard} matrices.  
 Let $G=\{g_{1},g_2,\ldots,g_{m}\}$ be a finite group and let   $H$ be an $n\times n$ matrix whose entries are elements of $G$. For convenience, we identify the group $G$ with the naturally embedded copy of $G$  in the group ring $\mathbb{Z}[G]=\{\sum^{m}_{i=1}a_{i}g_{i}|a_{i}\in\mathbb{Z} \}.$ A natural involution on $\mathbb{Z}[G]$, which we shall call conjugation, is defined by
 \begin{equation}
 \overline{(\sum a_{i}g_{i})}:=\sum a_{i}g^{-1}_{i}.
 \end{equation}
 Then, for the matrices with entries in $\mathbb{Z}[G]$, define an adjoint, $M^*=[m_{ij}]^*:=[\overline{m_{ji}}].$ That is the adjoint is performed by first taking the transpose of the matrix and then conjugating each entry.
We call that   an $n\times n$ matrix $H$ whose entries are in $G$ is a \textit{(group) generalized Hadamard matrix} if 
\begin{equation}
HH^*\equiv nI_n\mod \sum_{g \in G}{g}.
\end{equation}
For brevity, we say that such a matrix $H$ is a $\GH(n,G)$.  
 
For example, the following is a $\GH(5,C_{5}$), where $\zeta$ is a generator of the cyclic group $C_{5}$ of order 5.
 
$$ H=\left(
\begin{array}{ccccc}
1&\zeta &\zeta^{4}&\zeta^{4}&\zeta\\
\zeta&1&\zeta&\zeta^{4}&\zeta^{4}\\
\zeta^{4}&\zeta&1&\zeta&\zeta^{4}\\
\zeta^{4}&\zeta^{4}&\zeta&1&\zeta\\
\zeta&\zeta^{4}&\zeta^{4}&\zeta&1\\
\end{array}
\right)$$
Also we can say that a $\GH(n,G)$ can exist only if $n$ is a multiple of $|G|$. See also \cite{jungnickel1979difference} for more information about generalized Hadamard matrices.

It is known that a BH (resp. GH)  matrix can be transformed to an equivalent matrix consisting of 1s in the first row and column by dividing rows or columns, or by interchanging rows or columns. So, a BH (resp. GH) matrix is said to be \textit{normalized} if the first row and column consist entirely of 1s.

In this paper we first  study the parameters for code families where the rows of BH and modified BH matrices are assigned to be codewords.  
The minimum distance between the rows of normalized GH matrices has been studied widely, see \cite[Sections 4.4 and 9.4]{horadam2007hadamard}. In addition to minimum distance, the rank and kernel of the codes obtained from a GH matrix were recently studied in  \cite{dougherty2015ranks,dougherty2020rank}, and codes from a cocyclic GH matrix  and their equivalence to combinatorial difference sets are studied in \cite{armario2018quasi,armario2019generalized}. On the other hand, Greferath, McGuire and O'Sullivan \cite{greferath2006plotkin} showed that the codes obtained from BH matrices meet the Plotkin bound under any homogeneous weight. Stepanov \cite{stepanov2017nonlinear} also constructed codes obtained from modified BH matrices, which have parameters close to the Plotkin bound. 

In this paper, we  give a lower bound for the minimum distance of the codes   obtained from a normalized BH matrix and an amenable BH matrix in Theorem \ref{BHBCD} and \ref{prop:genhomweight}, respectively. Then, we prove their minimum distance under a homegeneous Gray map, in Corollary \ref{prop:homG1}. We also consider codes obtained from the image of the modified BH matrices under a non-homogeneous Gray map, and  we prove their minimum distances in Theorem \ref{thm:Nonhomtranslate}. We also show that the distances between the rows are same, hence the code is equidistant. Moreover, we show that the code meets the Plotkin bound in Corollary \ref{cor:plotkin}.

The paper is organized as follows. We study the generalized Plotkin bound in Section $2$ and generalized Gray map  in Section $3$. In Section $4$ we give some results on the minimum distance of codes obtained from BH and modified BH matrices. The codes obtained from the image of BH matrices under the non-homogeneous Gray map  are given  in Section $5$.

\section{The generalized Plotkin bound}

 Hadamard codes have been widely studied in the literature, see for instance \cite{horadam2007hadamard}.  Greferath et al. in \cite{greferath2006plotkin}   considered BH matrices and determined the parameters of the code obtained from normalized BH matrices, where they considered homogeneous weights. Their result is given below, but we first give the definition of a homogeneous weight.

 \begin{definition}\label{def:cc} \cite{greferath2006plotkin}
 A real-valued function $w$ on the finite ring $R$ is called a homogeneous weight if $w(0)=0$ and the following hold:\\
 (i) $Rx=Ry$ implies $w(x)=w(y)$ for all $x,y\in R$.\\
(ii) There exists a real number $\gamma$ such that 
 \begin{equation*}
     \sum_{y\in Rx}w(y)=\gamma|Rx|,  \forall x\in R \backslash \{0\},
 \end{equation*}
 where the number $\gamma$ is called the average value of $w$ on $R$. 
 \end{definition}
 
  For instance,  the
Lee weight $w_L$ on $\Z_4$ defined by $w_L(0) = 0, w_L(1) = w_L(3) = 1$ and $w_L(2) = 2$ is a homogeneous weight with the average value $\gamma=1$.

A $q$-ary $(n,M,d)$ code is defined to be a nonempty subset of $\Z_q^n$ of size $M$, where d is the minimum distance of the code determined by the distance function induced by a specified homogeneous weight on $\Z_q^n$. For an $(n,M,d)$ code with $d>\gamma n$, the \textit{generalized Plotkin bound} states that $M \leq d/(d - \gamma n)$,  and it is called \textit{Plotkin-optimal} when
\be \label{eqn:plotkin}
M > d/(d - \gamma n) - 1,
\ee
where $\gamma$ is the average value of the specified homogeneous weight.

 \begin{theorem}\label{thm:cc} \cite{greferath2006plotkin}
 Let $H=[\zeta_h^{a_{ij}}]$ be a normalized Butson-Hadamard matrix of type $\BH(n,h)$ for some primitive $h$-th root of unity $\zeta_h$.
 Then the code in $(\Z_{h})^{n-1}$ formed by taking the rows of $H'=[a_{ij}]$ and omitting the first coordinate is a code over $\Z_{h}$ with parameters $(n-1,n,\gamma n)$ that meets the Plotkin bound.
 
 \end{theorem}

\section{Generalized Gray maps}
In this section we consider two kinds of generalized Gray maps. At first, we give the generalized Gray map $G_1$, originated  in  \cite{heng2015generalized} (see Definition \ref{def:hom} below). Note that $G_1$ gives rise to a homogeneous weight. As Theorem \ref{thm:cc} states, codes obtained from rows of  BH matrices satisfy the Plotkin bound under any homogeneous weight \cite{greferath2006plotkin}. On the other hand, very little is known when the map is non-homogeneous. Hence, we consider a generalized Gray map $G_2$ which yields a non-homogeneous weight $w_2$ (see Definition \ref{def:non-hom}). We also prove some lemmas on $G_{2}$ in which we give the minimum distance of BH codes under the weight $w_{2}$ defined by $G_{2}$.  

\begin{definition}\label{def:hom}
Let $k$ be a positive integer, $u$ be any element of $\Z_{p^{k}}$ and  $u=\sum_{i=1}^{k}u_i p^{i-1} $ its $p$-ary expansion  for some $u_i \in \{0,1,\ldots,p-1\}$. The image of $u$ by a  \textit{Gray map $G_{1}$}   is defined to be the Boolean function on $GF(p)^{k-1}$: 
\be \nn
\begin{array}{lll}
        G_{1}(u):(y_{1},\ldots, y_{k-1})\rightarrow u_{k}+\sum _{i=1}^{k-1}u_{i}y_{i}.
\end{array}
\ee

\end{definition}

\begin{example}
Suppose that $k=3$ and $p=2$. Then the Gray map for $u$ is defined as follows:
\be \nn
\begin{array}{lll}
        G_1(u):(y_{1},y_{2})\rightarrow u_{3}+ u_{1}y_{1}+u_{2}y_{2},
\end{array}
\ee
where $(y_1,y_2)\in\mathbb{Z}^2_2$, in lexicographical order, are taken as follows: 
\begin{center}
\begin{tabular}{ c|c } 
 
 $y_{2}$ & $y_{1}$ \\ 
 \hline
 0 & 0 \\ 
 0 & 1 \\ 
 1 & 0 \\
 1 & 1 \\
\end{tabular}
\end{center}
If $u=6$, then its binary representation is equal to $(u_{3}u_{2}u_{1})=(110)$.  Therefore if $(y_{1},y_{2})=(0,0)$ then we get $G_{1}(6)=1$. If we continue like this we can get Boolean function $G_{1}(u)=(1100)$.
\end{example}

Let $w$ be the Hamming weight and $w_1$ be a weight  on $\Z_{p^k}$ defined by
$w_1(u) = w(G_1(u))$. We note that $w_{1}(0)=0$ and for any $u\neq0$
\be \nn 
      w_{1}(u):= \left\{ \begin{array}{cc}
      p^{k-1} - p^{k-2} &  \mbox{if } u \in \Z_{p^{k}} \backslash  \{p^{k-1},2p^{k-1},\ldots,(p-1)p^{k-1}\} \\
      p^{k-1}     & \mbox{otherwise.}
\end{array} \right.
\ee

Then one can easily deduce from Definition \ref{def:cc} that $w_1$ is a homogeneous weight. 
In what follows,  we list some properties of $G_1$.

\begin{enumerate}
\item[i.] Let $k$ be a positive integer and $u=\sum _{i=1}^{k}u_{i}2^{i-1}\in \Z_{2^{k}} $ for some $u_{i}\in\{0,1\}$. Then $G_1(u)$ is a homomorphism from $\Z_{2}^{k-1}$ to $\Z_{2}$.

\item[ii.] Let $a,b \in \Z_{{2}^{k}}$ such that $a\neq b$. Then 
\be \nn
d(G_1(a),G_1(b) )= \left\{ \begin{array}{cc} 2^{k-1}  &  \mbox{if } a-b= 2^{k-1} \\
2^{k-2}     & \mbox{otherwise}.
\end{array} \right.
\ee
Here $d(G_1(a),G_1(b))$ is the Hamming distance between $G_1(a)$ and $G_1(b)$.
\end{enumerate}

In this paper we also consider another Gray map which yields a non-homogeneous weight.

\begin{definition}\label{def:non-hom} \cite{yildiz2012generalization}
Let $p>2$ be a prime number and $k$ be a positive integer. A Gray map, denoted $G_2$, on $\Z_{p^k}$ is defined as follows:
\begin{itemize}
    \item[i.] If $u \leq p^{k-1}$, then $G_2(u) \in \Z_{p}^{p^{k-1}}$ has 1 in the first $u$ location and $0$'s elsewhere.
    \item[ii.] If $u \geq p^{k-1}+1$, then $G_2(u) = \bar{q}+G_2(r)$, where  $q$ and $r< p^{k-1}$ are positive integers such that 
$u=qp^{k-1}+r$
and $\bar{q}=(qqq\cdots qqq)$.
\end{itemize}

\end{definition}
Let $w_2$ be the  weight on $\Z_{{p}^{k}}$ defined  by $w_2(u) = w(G_2(u))$. Then we have
\be \nn 
      w_{2}(u):= \left\{ \begin{array}{ccc}
      u &  \mbox{if }& u\leq p^{k-1} \\
      p^{k-1}     & \mbox{if} &p^{k-1}\leq u\leq  p^{k}-p^{k-1}\\
     p^{k}-u & \mbox{if}& p^{k}-p^{k-1}<u\leq p^{k}-1.
\end{array} \right.
\ee

As the following example shows, $w_{2}$ is not homogeneous in general.

\begin{example}
Let $p=3$, $k=2$ and $R=\Z_{9}$. Now if we take $x=2$ and $y=8$, then $Rx=Ry$; but since $G_{2}(2)=(110)$, $G_{2}(8)=(002)$, $w_2(x)$ and $w_2(y)$ are not equal. Therefore $w_2$ is not homogeneous by Definition \ref{def:cc}.
\end{example}

Now we give four lemmas that we will use in the proof of Theorem \ref{thm:Nonhomtranslate}.
\begin{lemma}\label{lemma:numberzero}
Let $m,n\in \Z_{p^{k}}$ with $m,n \neq 0$. If either  $m,n<p^{k-1}$ and $m+n=p^{k-1}$ or $m>n>(p-1)p^{k-1}$ and $m-n=(p-1)p^{k-1}$, then the total number of $0$'s in $G_{2}(m)$ and $G_{2}(n)$ is $p^{k-1}.$   
\end{lemma}
\begin{proof}
If $m,n<p^{k-1}$ and $m+n=p^{k-1}$ then the result is clear from the definition of $G_{2}$.
Let $m>n$ and $m-n=(p-1)p^{k-1}.$ Since $m>n$ and $m=n+(p-1)p^{k-1}$ we see that $n<p^{k-1}$. Therefore, the number of $0$'s in $G_{2}(n)$ is equal to $p^{k-1}-n$.  On the other hand, $G_{2}(m)=G_{2}((p-1)p^{k-1}+n)$, and by definition, the number of $0$'s in $G_{2}(m)$ is $n.$ This completes the proof.  
\end{proof}
\begin{example}
Let $p=3$, $k=3$. 
 If we take $n=8$ and $m=n+(p-1)p^{k-1}=26$ then we can easily find that  $G_{2}(n)=G_{2}(8)=(111111110)$ and $G_{2}(m)=G_{2}(26)=(000000002)$. Therefore the total number of $0$'s in $G_{2}(m)$ and $G_{2}(n)$ is equal to $9=p^{k-1}.$
\end{example}

\begin{lemma}\label{dG}
Let $p$ be a prime number, $k\in \Z^{+}$, $i,j\in \Z_{p^{k}}$ and $|i-j|=n$. Then
\be \nn 
      d(G_{2}(i),G_{2}(j))= \left\{ \begin{array}{cc}
      n &  \mbox{if }  n< p^{k-1} \\
      p^{k-1}     & otherwise
    .
\end{array} \right.
\ee
\end{lemma}
\begin{proof}
First we consider $n<p^{k-1}$. Without loss of generality we take $i>j$. If $i\leq p^{k-1}$ then  $G_{2}(i)=(111\ldots 11000\ldots 00)$ by the definition of $G_{2}$, where it has the number $1$ in the first $i$ coordinates and the others are $0$. We also have $j<p^{k-1}$ and $j=i-n$ as $j<i.$ Therefore $G_{2}(j)=(111\ldots 11000\ldots 00)$ such that the first $i-n$ coordinates are $1$ and the others are $0$. So $d(G_{2}(i),G_{2}(j))=n$. Now if  $i>p^{k-1}$ then we can write $i=qp^{k-1}+r$ for some $q,r\in \Z^{+}$ such that $r<p^{k-1}.$  Then $G_{2}(i)=\overline{q}+G_{2}(r)$. Also $j=i-n $ and $n<p^{k-1}$. Now, if $n<r$ then we can say that $j=qp^{k-1}+r-n$ and $r-n<p^{k-1}$. Therefore $d(G_{2}(i),G_{2}(j))=n$. If $n\geq r$ then $j=(q-1)p^{k-1}+p^{k-1}-(n-r).$ Hence $G_{2}(i)=(q+1\quad q+1 \ldots q+1\underbrace {qq\ldots qq}_{p^{k-1}-r})$ and $G_{2}(j)=(qq\ldots qq \underbrace{q-1\quad q-1 \ldots q-1}_{n-r})$. Also we can say that $p^{k-1}-r>n-r$ since $n>r, n<p^{k-1}$. Therefore there are $p^{k-1}-r-n+r$ same coordinates in $G_{2}(i)$ and $G_{2}(j)$. So $d(G_{2}(i),G_{2}(j))=n$.

Now suppose that $j=qp^{k-1}+r, n=mp^{k-1}+r'$ and $r+r'=lp^{k-1}+r''$ for some $q,m,l,r,r',r''\in\Z^{+}$ such that $r,r',r''<p^{k-1}$. Then $i=(q+m+l)p^{k-1}+r''$ and since $i<p^{k}$ we must have $q+m+l<p$. It follows that $G(i)=\overline{q+m+l}+G_{2}(r'')$ and $G_{2}(j)=\overline{q}+G_{2}(r)$.  Since $r\neq r''$, we obtain $d(G_{2}(i),G_{2}(j))=p^{k-1}$. Also it is clear that if $j\leq p^{k-1}$, then $d(G_{2}(i),G_{2}(j))=p^{k-1}$. 
\end{proof} 

\begin{lemma}\label{totalzero}
Let $p$ be a prime number, $k\in\Z^{+}$, $\zeta$ be a primitive $p^{k}$-th root of unity and $a_i \in \Z^+$ for $i=1,2,\ldots,n$. 
If $\zeta ^{a_1}+\cdots +\zeta ^{a_n}=0 $, then the total number of zeros in $[G_{2}(a_{1}),G_{2}(a_{2}),\ldots ,G_{2}(a_{n})]$ is equal to $np^{k-2}$.
\end{lemma}
\begin{proof}
We know that $a_{i}$'s satisfying  $\zeta^{a_{1}}+\zeta^{a_{2}}+\cdots+\zeta^{a_{n}}=0$ will only be of the form $a_{i}=ip^{x}+j$, where $x<k, j<p^{x}$ and $i=1,2,\ldots,p^{k-x}$, see \cite[Theorem 3.3]{lam2000vanishing}.  Here we will give a proof only for the case $j=0$, since we will give the case where $j$ is nonzero in Corollary. \ref{lem:number_of_0} 
 We know that obtaining $0$ in the images $G_{2}(a_{i})$ for $i=1,2,\ldots,n$ is only possible for $1\leq a_{i}< p^{k-1}$ or $(p-1)p^{k-1}<a_{i}\leq p^{k}$. So the number of $0$ in  $G_{2}(a_{1}), G_{2}(a_{2}),\ldots, G_{2}(a_{p^{k-x}})$  are equal to $p^{k-1}-p^{x}, p^{k-1}-2p^{x},\ldots, (p^{k-1-x}-1)p^{x}$, respectively. Here there are $p^{k-1-x}-1$ elements. On the other hand, the number of zero in $G_{2}(a_{i})$ for $(p-1)p^{k-1}<a_{i}\leq p^{k}$ are $p^{x},2p^{x},\ldots ,(p^{k-1-x}-1)p^{k-1},p^{k-1}$, in order. Therefore, the total number of zeros in  $G_{2}(a_{1}),G_{2}(a_{2}),\linebreak \ldots ,G_{2}(a_{p^{k}-x})$  is 
 \begin{equation*}
     (p^{k-1-x}-1)p^{k-1}+p^{k-1}=p^{k-2}p^{k-x}=np^{k-2}.
 \end{equation*}
\end{proof}
\begin{example}
Let $p=3,k=3$ and $\zeta$ be the primitive $27-$th root of unity. Also we know that $\zeta^{3}+\zeta^{6}+\zeta^{9}+\zeta^{12}+\zeta^{15}+\zeta^{18}+\zeta^{21}+\zeta^{24}+\zeta^{27}=0$. Then $G_{2}(3)=(111000000), G_{2}(6)=(111111000), G_{2}(9)=(111111111), G_{2}(12)=(222111111), G_{2}(15)=(222222111), G_{2}(18)=(222222222), G_{2}(21)=(000111111)$, $ G_{2}(24)=(000000111),G_{2}(27)=(000000000)$. Total number of zeros in these images is equal to $np^{k-2}=27.$
\end{example}

\section{Butson-Hadamard codes}

In this section we will study codes obtained via assigning rows of normalized BH matrices and their translates as codewords. Namely, let $\zeta$ be a primitive $k$-th root of unity, $H=[\zeta^{a_{ij}}]$ be a normalized $\BH(n,k)$ matrix and $[a_{ij}]$ be the matrix with entries taken from the exponents of corresponding entries of $H$. Then codes are constructed from the rows of $[a_{ij}]$.

Codes from normalized generalized Hadamard matrices and their translates are widely studied, see \cite[Chapter 4]{horadam2007hadamard} and references therein. Besides, codes constructed from modified quaternary complex BH matrices are given  in \cite{stepanov2006new}, which are  close to the Plotkin bound. Another type modified $\BH$ matrices and the codes obtained from them are studied in  \cite{lee2006error}.

\begin{definition}
Let $n,k$ be positive integers and $\zeta$ be a primitive $k$-th root of unity, $H=[\zeta^{a_{ij}}]$ be a normalized $BH(n,k)$ matrix. Suppose that $N$ be the set of all $k$-th roots of unity.  If $R_{i}$ is the $i$-th row of H then, $uR_{i}$ is called a translate of $R_{i}$ for some $u\in N$.
\end{definition}

We use the definitions given in \cite[Definition 4.32]{horadam2007hadamard} and we take the Hamming weight for the codes obtained from BH matrices in the following theorem. 

\begin{theorem}\label{BHBCD}
Let $n,k$ be positive integers, $\zeta$ be a primitive $k$-th root of unity, $H=[\zeta^{a_{ij}}]$ be a normalized $\BH(n,k)$ matrix and $N$ be the set of all $k$-th roots of unity. Let $H'=[a_{ij}]$ and $l=\min  \{i\geq 2: i|k\}$. 
\begin{enumerate}
\item[i.] $ A_{k}$ the $k$-ary $(n-1,n,d_{A})$ code consisting of the rows of $H'$ with the first column deleted. Here  $d_{A}\geq n-\dfrac{n}{l}$. 
\item[ii.] $B_{k}$ the $k$-ary $(n-1,nk,d_{B})$ code consisting of the rows of exponent matrix of the translates $uH$,  $u\in N$ with first column deleted. Here  $d_{B}\geq n-\dfrac{n}{l}-1$
\item[iii.] $C_{k}$ the $k$-ary $(n,nk,d_{C})$ code consisting of the rows of exponent matrix of the translates $uH$, $u\in N$. Here $d_{C}\geq n-\dfrac{n}{l}$
\item[iv.] $D_{k}$, $n=k$, the $k$-ary $(n+1,n^{2},d_{D})$ code consisting of the rows of exponent matrix of $(uH)c$ for all $u\in N$ and any fixed noninitial column $c$ of $H$. Here $d_{D}\geq  n-\dfrac{n}{l}$.
\end{enumerate}
\end{theorem}
\begin{proof}
\begin{enumerate}
    \item [i.] 
We first note that $l$ is a prime number. If $\zeta$ is a primitive $k$-th  root of unity then $\zeta^{\frac{k}{l}}$ is  $l$-th root of unity. Then $(\zeta^{\frac{k}{l}})+(\zeta^{\frac{k}{l}})^2+\ldots+(\zeta^{\frac{k}{l}})^l=0$. This implies that we must have at least $l$ elements in order to get vanishing sum. On the other hand, since the rows of BH matrices are orthogonal, the number of same elements in two rows of $H$ must be smaller than $\frac{n}{l}$.
That is $n-d_A\leq \frac{n}{l}$. So we have $d_A\geq n-\frac{n}{l}$.
     \item [ii.] Let $u=\zeta^{m}$ be an element of $N$. Then it is clear that the codes $A_{k}$ and  shifting of $A_{k}$ with $m$ have same minimum distance. Now let $R_{i}$ and $R^u_{i}$ be the $i$-th row of $H$ and $uH$, respectively, with the first column deleted for $i=1,2,\ldots,n$. So $d(R_{i},R^u_{j})=n-1$ if $i=j$. On the other hand, as $R^u_{j}=uR_{j}$ we have $d(R_{i},R^u_{j}) = n-1 - \#\{u |u\in R_{i}^{-1}R_{j}^{u}\}$ for $i \neq j$. If we use the same argument in the proof of (i), we can say that $\#\{u |u\in R_{i}^{-1}R_{j}^{u}\} \le \frac{n}{l}$. Therefore $d_{B}\geq  n-\frac{n}{l}-1$.
    \item [iii.] $C_{k}$ has codewords of length $n$. Similar to the proof of (ii), we have $d_{C}\geq n-\frac{n}{l}$
    \item[iv.]  $D_{k}$ and $C_{k}$  are only one column different from each other. Therefore $d_{D}\geq n-\frac{n}{l}$.
\end{enumerate}
\end{proof}

\begin{example}\label{Ex:BH(9,10)}
Let $\zeta$ be a primitive $36$-th root of unity and 
$$H_1=\left(
\begin{array}{cccccccccccc}
1&1&1&1&1&1&1&1&1&1&1&1\\
1&\zeta^{12}&\zeta^{24}&\zeta^{28}&\zeta^{4}&\zeta^{16}&1&\zeta^{12}&\zeta^{24}&1&\zeta^{12}&\zeta^{24}\\
1&\zeta^{24}&\zeta^{12}&\zeta^{20}&\zeta^{8}&\zeta^{32}&1&\zeta^{24}&\zeta^{12}&1&\zeta^{24}&\zeta^{12}\\
1&\zeta^{27}&1&1&1&1&\zeta^{18}&\zeta^{9}&\zeta^{18}&\zeta^{18}&\zeta^{18}&\zeta^{18}\\
1&\zeta^{3}&\zeta^{24}&\zeta^{28}&\zeta^{4}&\zeta^{16}&\zeta^{18}&\zeta^{21}&\zeta^{6}&\zeta^{18}&\zeta^{30}&\zeta^{6}\\
1&\zeta^{15}&\zeta^{12}&\zeta^{20}&\zeta^{8}&\zeta^{32}&\zeta^{18}&\zeta^{33}&\zeta^{30}&\zeta^{18}&\zeta^{6}&\zeta^{30}\\
1&1&1&\zeta^{18}&\zeta^{18}&\zeta^{18}&\zeta^{9}&1&1&\zeta^{27}&\zeta^{18}&\zeta^{18}\\
1&\zeta^{12}&\zeta^{24}&\zeta^{10}&\zeta^{22}&\zeta^{34}&\zeta^{9}&\zeta^{12}&\zeta^{24}&\zeta^{27}&\zeta^{30}&\zeta^{6}\\
1&\zeta^{24}&\zeta^{12}&\zeta^{2}&\zeta^{26}&\zeta^{14}&\zeta^{9}&\zeta^{24}&\zeta^{12}&\zeta^{27}&\zeta^{6}&\zeta^{30}\\
1&\zeta^{27}&1&\zeta^{18}&\zeta^{18}&\zeta^{18}&\zeta^{27}&\zeta^{9}&\zeta^{18}&\zeta^{9}&1&1\\
1&\zeta^{3}&\zeta^{24}&\zeta^{10}&\zeta^{22}&\zeta^{34}&\zeta^{27}&\zeta^{21}&\zeta^{6}&\zeta^{9}&\zeta^{12}&\zeta^{24}\\
1&\zeta^{15}&\zeta^{12}&\zeta^{2}&\zeta^{26}&\zeta^{14}&\zeta^{27}&\zeta^{33}&\zeta^{30}&\zeta^{9}&\zeta^{24}&\zeta^{12}
\end{array}
\right)$$
be a $\BH(12,36)$ matrix. Here the minimum distance between the rows of $H_1$ is equal to $d_A=7$ by using SageMath \cite{sagemath}, therefore $d_A \geq n-\frac{n}{l}$ is satisfied, where $n=12$ and $l=min\{i \geq 2:i|36 \} =2$.
Similarly,
for a primitive $10$-th root of unity, $\zeta$, let
$$H_2=\left(
\begin{array}{ccccccccc}
1&1&1&1&1&1&1&1&1\\
1&\zeta^{5}&\zeta^{3}&\zeta^{3}&\zeta^{5}&\zeta^{9}&\zeta^{8}&\zeta^{7}&\zeta\\
1&\zeta^{4}&\zeta^{5}&\zeta^{7}&\zeta&\zeta^{3}&\zeta^{5}&\zeta^{9}&\zeta^{9}\\
1&\zeta^{3}&\zeta^{7}&\zeta^{5}&\zeta&\zeta^{8}&\zeta^{9}&\zeta^{3}&\zeta^{5}\\
1&\zeta^{9}&\zeta&\zeta^{5}&\zeta^{5}&\zeta^{3}&\zeta^{7}&\zeta^{2}&\zeta^{7}\\
1&\zeta^{9}&\zeta^{5}&\zeta&\zeta^{3}&\zeta^{5}&\zeta&\zeta^{7}&\zeta^{6}\\
1&\zeta&\zeta^{7}&\zeta^{9}&\zeta^{6}&\zeta&\zeta^{5}&\zeta^{5}&\zeta^{3}\\
1&\zeta^{7}&\zeta^{9}&\zeta^{4}&\zeta^{9}&\zeta^{5}&\zeta^{3}&\zeta^{5}&\zeta\\
1&\zeta^{5}&\zeta^{2}&\zeta^{9}&\zeta^{7}&\zeta^{7}&\zeta^{3}&\zeta&\zeta^{5}
\end{array}
\right)$$
be a $\BH(9,10)$ matrix. It can be verified that  the minimum  distance between the rows of $H_2$ is equal to $d_A=7$ and satisfies $d_A \geq n-\frac{n}{l}$, where $n=9$ and $l=2$. 
Now consider $B_{k}$. In this case, we have $A_{k}$ and its translates as codewords. If we calculate the minimum distance by SageMath, we get $d_{B}=5$ and so $d_B \geq n(1-\frac{1}{l}) -1$. Again, by using SageMath, we can get that the minimum distance for $C_{k}$ is $d_{C}=6$, which satisfies $d_{C} \geq n-\frac{n}{l}$.

Let $\zeta$ be a primitive $6$-th root of unity. Then 
$$H_2=\left(
\begin{array}{cccccc}
1&1&1&1&1&1\\
1&\zeta&\zeta^{2}&\zeta^{3}&\zeta^{4}&\zeta^{5}\\
1&\zeta^{2}&\zeta^{4}&1&\zeta^{2}&\zeta^{4}\\
1&\zeta^{3}&1&\zeta^{3}&1&\zeta^{3}\\
1&\zeta^{4}&\zeta^{2}&1&\zeta^{4}&\zeta^{2}\\
1&\zeta^{5}&\zeta^{4}&\zeta^{3}&\zeta^{2}&\zeta\\
\end{array}
\right)$$
is a $\BH(6,6)$ matrix. If we calculate the minimum distance for $D_{k}$ by SageMath, we  get $d_D=3$, so $d_D \geq n(1-\frac{1}{l})$ holds.

\end{example}

\begin{remark} We note that it is easy to get a code with larger alphabet size without decreasing the minimum distance by combining two BH matrices under Chinese Remainder Theorem.
Let $h,k$ be distinct prime numbers, $\zeta$, $\alpha$ and $\beta$ be primitive $h$-th, $k$-th and $hk$-th  roots of unities, respectively. Also  $H_{1}=[\zeta^{a_{ij}}]$ and $H_{2}=[\alpha^{b_{ij}}]$ be normalized $\BH(n,h)$ and $\BH(n,k)$ matrices, respectively. Let $d_{1}$ and $d_{2}$ be the minimum distance between the rows of $[a_{ij}]$ and $[b_{ij}]$, respectively. Let $H=[\beta^{c_{ij}}]$ be the normalized matrix obtained from $H_{1}$ and $H_{2}$  such that $c_{ij}\equiv a_{ij} \mod h$ and $c_{ij}\equiv b_{ij} \mod k$. Then the minimum distance of rows of $[c_{ij}]$ satisfies $d \geq \max\{d_{1},d_{2}\}$. 
\end{remark}

\begin{proposition}\label{theo:amenable BH}
Let $R$ be a Frobenius ring and let $A=[a_{ij}]$ be a $n\times n$ matrix over $R$. Suppose that for any $1\leq k,l\leq n$ with $k\neq l$, the set 
\be \label{eq:disjoint union}
 R_{k}-R_{l}:=\{a_{k1}-a_{l1},a_{k2}-a_{l2},\ldots,a_{kn}-a_{ln}\},
\ee
is a disjoint union of the cosets of right or left ideals of $R$. Here $R_{k}$ and $R_{l}$ are the $k$-th and $l$-th rows of $A$, respectively. Then for any generating character $\chi$ of $R$, the matrix $H=[\chi(a_{ij})]$ is a Butson-Hadamard matrix.
\end{proposition}
\begin{proof}
Suppose that for $1\leq i,j\leq n$ with $i\neq j$, $S_{i}$ and $S_{j}$ be the distinct rows of $H$. Then $S_{i}=(\chi(a_{i1}), \chi(a_{i2}),\ldots,\chi(a_{in}))$ and $S_{j}=(\chi(a_{j1}), \chi(a_{j2}),\ldots,\chi(a_{jn}))$. Therefore $S_{i}S_{j}^{*}=\displaystyle\sum_{k=1}^{n}\chi(a_{ik}-a_{jk})$. We know that the sum of the images under $\chi$ over a right (or left) ideal of $R$ is equal to $0$. Then for a fixed element $a\in R$ and a right (or left) ideal $I$ of $R$, we have $\displaystyle\sum_{i\in I}\chi(a+i)=\displaystyle\sum_{i\in I}\chi(a)\chi(i)=\chi(a)\sum_{i\in I}\chi(i)=0$; so we can say that $S_{i}S_{j}^{*}=\displaystyle\sum_{k=1}^{n}\chi(a_{ik}-a_{jk})=0.$ 

  If $i=j$ then $\displaystyle\sum_{k=1}^{n}\chi(a_{ik}-a_{jk})=\displaystyle\sum_{k=1}^{n}\chi(a_{ik}-a_{jk})=n$. Therefore $H$ is a Butson-Hadamard matrix. 
\end{proof}
\begin{definition}\label{def:amenable BH}
Let $R$ be a Frobenius ring and  $A=[a_{ij}]$ be an $n\times n$ matrix over $R$. Also suppose that $A$  satisfies (\ref{eq:disjoint union}). Then for any generating character $\chi$ of $R$, the Butson-Hadamard matrix $H=[\chi(a_{ij})]$ is called an \emph{amenable} Butson-Hadamard matrix.   
\end{definition}
\begin{example}
Suppose that $R$ is a finite Frobenius ring of characteristic $t$ and $\chi$ is a generating character for $R$, $B:L\times M\rightarrow R$ be a nondegenerate bilinear pairing for the finite left $R$-module $L$ and finite right $R$-module $M$. Then $H=[\chi(B(x,y))]_{L\times M}$ is a BH(t,m) Butson-Hadamard matrix, where $m=|L|=|M|$  by \cite[Theorem 10]{mcguire2009cocyclic}. All matrices created in this way are amenable Butson-Hadamard matrices since the elements in the row differences in $[B(x,y)]_{L\times M}$ form left ideals of $R$.
\end{example}

  Bilinear pairings give us amenable Butson-Hadamard matrices, but not all amenable $BH$ matrices have to be of this type. The matrix below is a counterexample.
\begin{example}
Let $R=\Z_{6}$ and $A$ be a $6\times 6$ matrix over $R$ given as below.
$$A=\left(
\begin{array}{ccccccc}
0&0&0&0&0&0&0\\
0&0&0&2&3&3&4\\
0&1&3&5&0&3&3\\
0&2&5&4&3&0&2\\
0&3&4&2&5&2&0\\
0&4&2&0&3&4&1\\
0&4&2&3&1&0&4
\end{array}
\right)$$
\end{example}
It is readily checked that all pairwise row differences of $A$ are the unions of the cosets of $3\Z_{6}$ and $2\Z_{6}$.  Thus, if  $\chi:R\longrightarrow \mathbb {C}, \chi(a)=\zeta^{a}$, where $\zeta$ is a primitive $6$-th root of unity, then $H=[\chi(a_{ij})]$ is an amenable $BH(7,6)$ matrix. 

\begin{example}\label{ex:prime_power_bh_amenable)}
All matrices of type $BH(n,p^{k})$ are amenable $BH$ matrices where $p$ is a prime number and $k,n\in\Z^{+}$, because every row difference in the additive representation of a $BH(n,p^{k})$ forms up a union of cosets of ideals $p^{i}Z_{p^{k}}$  for some $1\leq i\leq k-1$, see  \cite[Theorem 3.3]{lam2000vanishing}. 
\end{example}
\begin{remark}
The authors of this article do not know the existence of a Butson-Hadamard matrix which is not amenable.
\end{remark}

In Theorem \ref{BHBCD} we have given lower bounds for the minimum distances of codes defined via rows of BH matrices. But in the following theorem we have obtained equality for these minimum distances by taking BH matrices as amenable.  

 Now suppose that $w$ is a homogeneous weight, $l=\max\{w(a) : a\in \Z_{k}\}$, $l'=\min\{w(a) : a\in\Z_{k}\backslash\{0\}\}$ and $\gamma$ is the average value of $w$.
\begin{theorem}
\label{prop:genhomweight}
Let $n$ and $k$ be positive integers, $H=[\zeta^{a_{ij}}]$ be a normalized amenable $\BH(n,k)$ matrix, $H'=[a_{ij}]$ and $N$ be the set of all $k$-th roots of unity.  Also suppose that   $A_{k}, B_{k}, C_{k}$ and $D_{k}$ are defined  as in Theorem $\ref{BHBCD}$. Then $d_{A_{k}}=\gamma n, d_{B_{k}}=\gamma n-l,  d_{C_{k}}=\gamma n$ and $d_{D_{k}}\in\{\gamma n, \gamma n+l'\}$.
\end{theorem}
\begin{proof}
We know that $d_{H}=\gamma n$ by Theorem \ref{thm:cc}. Now consider $B_{k}$. Let $R_{i}, R_{j}$ be the $i$-th  and $j$-th  rows of $H$, respectively. Also suppose that $R'_{j}$ is the $j$-th row of $\zeta^{u}H$ such that $1\leq i,j\leq n$ and $i\neq j, u\in Z_{t}$. So $R'_{j}=\zeta^{u}R_{j}$. If we use the same argument in the proof of \ref{thm:cc} we can say that $f_{r+u}=|\{k\in \{1,\ldots, n\} : \log_{\zeta}(R_{ik}R^{-1}_{jk})=r+u \}|/n=f_{r}$ for all $r,u\in \Z_{t}$. Here $\log_{\zeta}(R_{ik}R^{-1}_{jk})$ is the $k$-th coordinate of $R_{i}R_{j}^{-1}.$ Therefore the proof of the Theorem \ref{def:cc}  is also provided here as well. So $\sum_{k=1}^{k=n}(w(R_{ik}R^{-1}_{jk}))=\gamma n$. But here  since the first coordinates of $R_{i}$ and $R'_{j}=uR_{j}$ will always be different if we delete the first coordinate, the minimum distance decrease by $l$ . After all this we get $d_{B_{k}}=\gamma n-l$.

 The $C_{k}$ code differs from $B_{k}$ only in the first coordinate. Therefore, from the results we obtained above we get $d_{C_{k}}=\gamma n.$
  
 Finally we consider $D_{k}$ code. Since $D_{k}$ is obtained by adding  a column to $C_{k}$ we can say that $d_{D_{k}}\in\{\gamma n, \gamma n+l'\}$.
\end{proof}

As noted Example \ref{ex:prime_power_bh_amenable)}, all $BH$ matrices in the following corollary are amenable.
\begin{corollary}\label{prop:homG1}
Let  $p>2$ be a prime number, $ k\in \Z^{+}$ such that $k\geq 2$ and $\zeta $ be a primitive $p^{k}$-th  root of unity,    $H=[\zeta^{a_{ij}}]$  be an $n\times n$ normalized Butson-Hadamard matrix and $N$ be the set of all $p^{k}$-th roots of unity.  Also let $H'=[a_{ij}]$ 
and define $A_{k},B_{k},C_{k},D_{k}$ codes 
as in Theorem \ref{BHBCD}. Then,
$d_{G_{1}(A_{k})}=n(p-1)p^{k-2} $, $d_{G_{1}(B_{k})}=n(p-1)p^{k-2}-p^{k-1}$, $d_{G_{1}(C_{k})}=n(p-1)p^{k-2}$, $d_{G_{1}(D_{k})}\in\{p^{2k-2}(p-1),p^{k-2}(p^{k+1}-p^{k}+p-1)\}$.
\end{corollary}
\begin{proof}

$d_{G_{1}(A_{k})}$ is a consequence of Theorem \ref{thm:cc}.  We need to find the $\gamma$ average value of $w$. Now since $G_{1}$ is balanced except for the elements of  $L=\{p^{k-1},2p^{k-1},\ldots, p^{k}\}$. Namely for every $1\leq \alpha \leq p^k$ such that $\alpha \notin L$ there are equal number of $0,1,\ldots ,p-1$ in  $G_{1}(\alpha)$. For any $\beta\in L\backslash \{0\},  w(\beta)=p^{k-1}$.   Therefore $l=p^{k-1}-p^{k-2}, l'=p^{k-1}$ and
\begin{equation*}
    \gamma=\frac{(p-1)p^{k-1}+(p^{k}-p)(p^{k-1}-p^{k-2})}{p^{k}}=p^{k-1}-p^{k-2}=(p-1)p^{k-2}.
\end{equation*}
So the result is obtained when the $\gamma$ is written instead of the Theorem \ref{prop:genhomweight}.
\end{proof}

Now we will give an example of  Theorem \ref{prop:genhomweight}.
\begin{example} \label{ex:homG1}
Let $p=3$,  $k=2$ and $\zeta$ be the primitive $9$-th root of unity. Then 
$$H=\left(
\begin{array}{ccccccccc}
1&1&1&1&1&1&1&1&1\\
1&\zeta &\zeta^{2}&\zeta^{3}&\zeta^{4}&\zeta^{5}&\zeta^{6}&\zeta^{7}&\zeta^{8}\\
1&\zeta^{2}&\zeta^{4}&\zeta^{6}&\zeta^{8}&\zeta&\zeta^{3}&\zeta^{5}&\zeta^{7}\\
1&\zeta^{3}&\zeta^{6}&1&\zeta^{3}&\zeta^{6}&1&\zeta^{3}&\zeta^{6}\\
1&\zeta^{4}&\zeta^{8}&\zeta^{3}&\zeta^{7}&\zeta^{2}&\zeta^{6}&\zeta&\zeta^{5}\\
1&\zeta^{5}&\zeta&\zeta^{6}&\zeta^{2}&\zeta^{7}&\zeta^{3}&\zeta^{8}&\zeta^{4}\\
1&\zeta^{6}&\zeta^{3}&1&\zeta^{6}&\zeta^{3}&1&\zeta^{6}&\zeta^{3}\\
1&\zeta^{7}&\zeta^{5}&\zeta^{3}&\zeta&\zeta^{8}&\zeta^{6}&\zeta^{4}&\zeta^{2}\\
1&\zeta^{8}&\zeta^{7}&\zeta^{6}&\zeta^{5}&\zeta^{4}&\zeta^{3}&\zeta^{2}&\zeta
\end{array}
\right)$$
is an amenable $\BH(9,9)$ matrix. Then 
$$G_{1}(H')=\left(
\begin{array}{ccccccccccccccccccccccccccc}
0&0&0&0&0&0&0&0&0&0&0&0&0&0&0&0&0&0&0&0&0&0&0&0&0&0&0\\
0&0&0&0&1&2&0&2&1&1&1&1&1&2&0&1&0&2&2&2&2&2&0&1&2&1&0\\
0&0&0&0&2&1&1&2&0&2&2&2&2&1&0&0&1&2&1&1&1&1&0&2&2&0&1\\
0&0&0&1&1&1&2&2&2&0&0&0&1&1&1&2&2&2&0&0&0&1&1&1&2&2&2\\
0&0&0&1&2&0&2&1&0&1&1&1&2&0&1&0&2&1&2&2&2&0&1&2&1&0&2\\
0&0&0&1&0&2&0&1&2&2&2&2&0&2&1&2&0&1&1&1&1&2&1&0&1&2&0\\
0&0&0&2&2&2&1&1&1&0&0&0&2&2&2&1&1&1&0&0&0&2&2&2&1&1&1\\
0&0&0&2&0&1&1&0&2&1&1&1&0&1&2&2&1&0&2&2&2&1&2&0&0&2&1\\
0&0&0&2&1&0&2&0&1&2&2&2&1&0&2&1&2&0&1&1&1&0&2&1&0&1&2
\end{array}
\right).$$
If we calculate the average weight we get $\gamma=2$. 
By using SageMath \cite{sagemath}, we obtain that  $d_{G_{1}(A_{k})}=18=\gamma n$, $d_{G_{1}(B_{k})}=15=\gamma n-l, d_{G_{1}(C_{k})}=18=\gamma n$. Here  $l=\max\{w(a) : a\in \Z_{k}\setminus\{0\}\}=3$. Let $c$ be the $4$-th column of $H'$. Then we get $d_{G_{1}(D_{k})}=18$. These minimum distances comply with the results in Theorem \ref{prop:genhomweight}.
\end{example}

If $w$ is a homogeneous weight with the average value $\gamma$ over a Frobenius ring $R$, then $\sum_{r\in y+I}w(r)=\gamma |I|$ for every nonzero ideal $I$ of $R$, see \cite[Proposition 2.6]{greferath2006plotkin}, a property which plays a crucial role in proving that some linear codes produced from a bimodule over $R$ by using a non-degenerated bilinear form meet the Plotkin bound \cite[Theorem 4.3]{greferath2006plotkin}. Note that this important property of a homogeneous weight is shared also by the non-homogeneous weight $w_2$ induced by the Gray map $G_2$ with $\gamma=(p-1)p^{k-2}$, see Corollary \ref{cor:w_2 is quasi-homogeneous} below. In the following definition we formalize such a weight.
\begin{definition}
 A weight $w$ on the finite ring $R$ is called  \emph{quasi-homogeneous} if $w(0)=0$ and there exists a real number $\gamma$ such that 
 \[
     \sum_{r\in a+I}w(r)=\gamma|I|,  
 \]
 for each nonzero ideal $I$ of $R$. The number $\gamma$ is called the \emph{average value} of $w$. 
 \end{definition}

 \begin{proposition}\label{prop:quasi-homogeneous weight}
 Let $k$ be a positive integer and $\gamma$ be a positive real number. Define a mapping $w:\mathbb{Z}_{p^k}\rightarrow\mathbb{R}$ by
 \[
 w(u)=
 \begin{cases}
 \frac{\gamma}{p^{k-2}(p-1)}u, &\quad \text{if } u_{k-1}=0 \\
 \frac{\gamma p}{p-1}, &\quad \text{if } 0<u_{k-1}\le p-2\\
 \frac{\gamma p^2}{p-1}-\frac{\gamma}{p^{k-2}(p-1)}u, & \quad \text{if } u_{k-1}=p-1
 \end{cases}
 \]
for all $u\in \mathbb{Z}_{p^k}$ with the  $p$-ary expansion $u=u_0+u_1p+\cdots+ u_{k-1}p^{k-1}$. Then $w$ is a quasi-homogeneous weight on $\mathbb{Z}_{p^k}$ with the average value $\gamma$.
 \end{proposition}
 \begin{proof}
 Fix $0\le s\le k-1$. We show that for any $a\in \mathbb{Z}_{p^k}$, \[\sum_{u\in a+\langle p^s\rangle}w(u)=\gamma p^{k-s},\]which completes the proof. Note that it is enough to assume  $0\le a\le p^s-1$. Note also that for any $u\in a+\langle p^s\rangle$, there exists a unique $n\in\mathbb{Z}_{p^k}$ such that $0\le n\le p^{k-s}-1$ and $u=a+np^s$. Write $a=a_0+a_1p+\cdots+a_{s-1}p^{s-1}$ and $n=n_0+n_1p+\cdots+n_{k-s-1}p^{k-s-1}$, where $0\le a_i,n_j\le p-1$ for each $i=0,\ldots,s-1$ and $j=0,\ldots,k-s-1$. Hence, a typical element $u$ of $a+\langle p^s\rangle$ can be written uniquely as \[u=a_0+a_1p+\cdots+a_{s-1}p^{s-1}+n_0p^s+n_1p^{s+1}+\cdots+n_{k-s-2}p^{k-2}+n_{k-s-1}p^{k-1},\]where the $a_i$'s and $n_j$'s all lie in $\{0,\ldots p-1\}$. We separate the sum of $w(u)$'s, where $u$ ranges over $a+\langle p^s\rangle$, into three sums as follows:
 
 (I) The sum of the $w(u)$'s for $u\in a+\langle p^s\rangle$ with $n_{k-s-1}=0$, i.e., \[\frac{\gamma}{p^{k-2}(p-1)}\left[p^{k-s-1}a+p^s\frac{p^{k-s-1}(p^{k-s-1}-1)}{2}\right];\]
 
 (II) the sum of the $w(u)$'s for $u\in a+\langle p^s\rangle$ with $0\le n_{k-s-1}\le p-2$, i.e., \[\frac{\gamma p(p-2)p^{k-s-1}}{p-1};\]
 
 (III) the sum of the $w(u)$'s for $u\in a+\langle p^s\rangle$ with $n_{k-s-1}=p-1$, i.e., \[\frac{\gamma p^2p^{k-s-1}}{p-1}-\frac{\gamma}{p^{k-2}(p-1)}\left[p^{k-s-1}a+\frac{p^sp^{k-s-1}(p^{k-s-1}-1)}{2}+p^{k-s-1}(p-1)p^{k-1}\right].\] It is now easy to see that the numbers in (I)--(III) sum up to $\gamma p^{k-s}$.
 \end{proof}

\begin{corollary}\label{cor:w_2 is quasi-homogeneous}
 The weight $w_2$ on $\mathbb{Z}_{p^k}$, where $p>2$ is a prime number and $k\in \mathbb{Z}^+$, is a quasi-homogeneous weight with the average value $\gamma=p^{k-2}(p-1)$. 
 \end{corollary}
 \begin{proof}
 Notice that $w_2$ coincides with the weight defined in Proposition \ref{prop:quasi-homogeneous weight} with $\gamma=p^{k-2}(p-1)$.
 \end{proof}
 \begin{remark}
 In view of the proof of Proposition \ref{prop:homG1}, we see that both weights $w_1$ and $w_2$ share the same average value $\gamma=(p-1)p^{k-2}$. In fact, $w_1$ is the only homogeneous weight on $\mathbb{Z}_{p^k}$ with the average value $\gamma=(p-1)p^{k-2}$, see Theorem 2.2 in \cite{greferath2006plotkin}. It follows that our consideration of quasi-homogeneous weights enables us to work  with more useful weights with the same common average value and flourishes our repository of  weights in this manner.
 \end{remark}
 \begin{corollary}\label{lem:number_of_0}
 Let $p>2$ be a prime number, $k$ be a positive integer, and $i \in \{1,2,\ldots,k-1\}$. Given any $j\in\mathbb{Z}_{p^k}$, the total number of $0$'s in the elements of the set  $\{G_2(u):u\in j+\langle p^i\rangle\}$ is independent of $j$ and equals $p^{2k-i-2}$.
 \end{corollary}
 \begin{proof}
 By Corollary \ref{cor:w_2 is quasi-homogeneous}, the total number of $0$'s in the elements of the set  $\{G_2(u):u\in j+\langle p^i\rangle\}$ is \begin{align}
 p^{k-1}p^{k-i}-\sum_{u\in j+\langle p^i\rangle}w_2(u)&=p^{2k-i-1}-p^{k-2}(p-1)p^{k-i}\nonumber \\
 &=p^{2k-i-2} \nonumber
 \end{align}
 \end{proof}
 \begin{example}

Choose $p=5$, $k=2$, $i=1$ and $j=3$. The total number of $0$'s in $G_{2}(5)=(11111), G_{2}(10)=(22222), G_{2}(15)=(33333), G_{2}(20)=(44444), G_{2}(25)=(00000)$ is $p^{2k-i-2}=5$.
On the other hand, the total number of $0$'s in   $G_{2}(8)=(22211), G_{2}(13)=(33322), G_{2}(18)=(44433), G_{2}(23)=(00044), G_{2}(3)=(11100)$ is $p^{2k-i-2}=5$. Hence, the total number of $0$'s in both sets is the same.
    
\end{example}
\begin{proposition}\label{prop:codewithquasihom}
Let $p>2$ be a prime number, $k\in \Z^{+}$ such that $k\geq 2$ and $\zeta$ be a primitive $p^{k}$-th root of unity, $H=[\zeta^{a_{ij}}]$ be an $n\times n$ normalized Butson-Hadamard matrix, $H'=[a_{ij}]$ and $w'$ be a quasi-homogeneous weight with the average value $\gamma$.
Then the minimum distance of the code consisting the rows of $H'$ with the first column deleted is equal to $\gamma n$.
\end{proposition}
\begin{proof}
Since $H$ is a Butson-Hadamard matrix we know that there exists $m\in\Z^{+}$ such that $n=mp$. We will the proof by induction on $m$.
 
For $m=1$ we have a $BH(p,p)$ matrix.Let $R_{i}, R_{j}$ be the $i$-th and $j$-th rows of $H'$, respectively.  Hence the elements of $R_{i}-R_{j}$ consists of the elements of $p^{k-1}\Z_{p^{k}}$ or the cosets of this ideal. More clearly $R_{i}-R_{j}= \{p^{k-1},2p^{k-1},\ldots, p^{k}\}$  or $\{p^{k-1}+i,2p^{k-1}+i,\ldots, p^{k}+i\}$ for all $1\leq i\leq p^{k-1}-1$. So for $I=p^{k-1}\Z_{p^{k}}$ we can say that $\displaystyle \sum_{r\in I+y}w'(r)=\gamma p=\gamma n$. Now let $m>1$ be given and suppose that the theorem is true for all $m=l$ with $l\in\Z^{+}$. Let us show the theorem is true for $m=l+1$. Suppose that $m=l+1$. We know that $d=\gamma n$ for $n=l$. If we take $n=l+1$ this differs from the case $n=l$ only $p$ elements.The $p$ elements again only come from the $p^{k-1}\Z_{p^{k}}$ ideal or the cosets of this ideal.  Therefore $d=\gamma lp+\gamma p=\gamma p(l+1)=\gamma n$. This ends the proof.
\end{proof}
\begin{theorem}\label{thm:Nonhomtranslate}
Let  $p>2$ be a prime number, $ k\in \Z^{+}$ such that $k\geq 2$ and $\zeta $ be a primitive $p^{k}$-th  root of unity,    $H=[\zeta^{a_{ij}}]$  be an $n\times n$ normalized Butson-Hadamard matrix and $N$ be the set of all $p^{k}$-th roots of unity.  Also let $H'=[a_{ij}]$  
and define $A_{k},B_{k},C_{k},D_{k}$ codes
as in Theorem \ref{BHBCD}. Then, 

$d_{G_{2}(A_{k})}= n(p-1)p^{k-2}$, $d_{G_{2}(B_{k})}=n-1$, $d_{G_{2}(C_{k})}=n$, $d_{G_{2}(D_{k})}\in \{ n,n+1\}$.
\end{theorem}

\begin{proof} 
We can say that $d_{G_{2}(A_{k})}=\gamma n$ by Proposition \ref{prop:codewithquasihom} and Corollary \ref{cor:w_2 is quasi-homogeneous} we get $d_{G_{2}(A_{k})}=n(p-1)p^{k-2}$.

Now consider $G_{2}(B_{k})$. Let $R_{i}=[a_{i2} \ldots a_{in}]$ and $R'_{i}=[a'_{i2} \ldots a'_{in}]$ be the $i$-th row of the  exponent matrix of $H$ and $\zeta H$ with deleted first column, respectively. Then $|a_{ij}-a'_{ij}|=1$, so  we get  $d(G_{2}(a_{ij}),G_{2}(a'_{ij}))=1$ for $j=2,3,\ldots,n$ by Lemma \ref{dG}. Therefore the  distance between $G_{2}(R_{i})$ and $G_{2}(R'_{i})$   is equal to $n-1.$ Now suppose that  $R'_{j}$ be the $j$-th row of the exponent matrix of $uH$ for an arbitrary $u\in N$. By Corollary \ref{lem:number_of_0}  and Lemma \ref{totalzero}, we can say 
\begin{equation*}
  d(G_{2}(R_{i}),G_{2}(R'_{j}))=(n-1)p^{k-1}-np^{k-2}=p^{k-2}(n(p-1)-1)
\end{equation*}
for $i\neq j$.
 Since $n-1<p^{k-2}(n(p-1)-1)$,  the minimum distance of $G_{2}(B_{k})$ is $d_{G_{2}(B_{k})}=n-1$. 
 
We next consider   $G_{2}(C_{k})$. Its codewords are constructed from the rows of the exponent matrix of $uH$ for $u \in N$. Similar to (ii), by considering $i$-th rows $R_i$ and $R'_i$ of $H$ and $\zeta H$ respectively,  we have $d(G_{2}(R_{i}), G_{2}(R'_{i}))=n$. And so, $d_{G_{2}(C_{k})}=n$.

Finally the code $G_2({D_{k}})$ is the Gray image of the rows of $uH$ concatenated with a non-initial column $c$ of  $H$. So, by using (iii), we can say that $d_{G_2({D_{k}})}$ is either $n$ or $n+1$.
\end{proof}

Now we can directly obtain a Plotkin optimal code form $G_2$ image of $A_k$ in the following corollary. 
\begin{corollary} \label{cor:plotkin}
The code $G_2(A_k)$ is a Plotkin-optimal $p$-ary nonlinear $((n-1)p^{k-1},n,n(p-1)p^{k-2})$ code under the non-homogeneous $w_2$ weight. Similarly $G_{2}(B_{k}), G_{2}(C_{k})$ and $G_{2}(D_{k})$ are Plotkin-optimal $p$-ary nonlinear $((n-1)p^{k-1},np^{k},n-1), (np^{k-1},np^{k},n), ((n+1)p^{k-1},np^{k},d_{G_{2}(D_{k})})$ codes under the non-homogeneous $w_{2}$ weight where $d_{G_{2}(D_{k})}$ is in $\{n,n+1\}$.
\end{corollary}

\begin{remark}
Theorem \ref{thm:Nonhomtranslate} is not a special case of Proposition \ref{prop:homG1}. Because the weight we use in Theorem \ref{thm:Nonhomtranslate} is non-homogeneous.  We know that the sum of elements of a row in any normalized Butson-Hadamard matrix is equal to $0$. Therefore, the code $G_2(A_k)$  is constant weight and equidistant code. 
\end{remark}
 \begin{example}\label{BH(9,9)}
Let $H=[\zeta^{a_{ij}}]$ be a $\BH(9,9)$ matrix defined as in Example \ref{ex:homG1}. Then 

$$H'=[G_{2}(a_{ij})]=\left(
\begin{array}{ccccccccccccccccccccccccccc}
0&0&0&0&0&0&0&0&0&0&0&0&0&0&0&0&0&0&0&0&0&0&0&0&0&0&0\\
0&0&0&1&0&0&1&1&0&1&1&1&2&1&1&2&2&1&2&2&2&0&2&2&0&0&2\\
0&0&0&1&1&0&2&1&1&2&2&2&0&0&2&1&0&0&1&1&1&2&2&1&0&2&2\\
0&0&0&1&1&1&2&2&2&0&0&0&1&1&1&2&2&2&0&0&0&1&1&1&2&2&2\\
0&0&0&2&1&1&0&0&2&1&1&1&0&2&2&1&1&0&2&2&2&1&0&0&2&2&1\\
0&0&0&2&2&1&1&0&0&2&2&2&1&1&0&0&2&2&1&1&1&0&0&2&2&1&1\\
0&0&0&2&2&2&1&1&1&0&0&0&2&2&2&1&1&1&0&0&0&2&2&2&1&1&1\\
0&0&0&0&2&2&2&2&1&1&1&1&1&0&0&0&0&2&2&2&2&2&1&1&1&1&0\\
0&0&0&0&0&2&0&2&2&2&2&2&2&2&1&2&1&1&1&1&1&1&1&0&1&0&0
\end{array}
\right).$$
$G_{2}(A_{k})$ is consisting of the rows of $H'$ with deleted first 3 columns as codewords. By SageMath \cite{sagemath}, we have    $d_{A_{k}}= 18$ which satisfies $d_{A_{k}}=n(p-1)p^{k-2}$. Similarly, by using SageMath, we get that $d_{G_{2}(B_{k})}=8=n-1$, $d_{G_{2}(C_{k})}=9=n$ and $d_{G_{2}(D_{k})}=9=n$, where $D_{k}$ is obtained by concatenating the $4$-th column of $H$ to $C_k$.
\end{example}

\section{Conclusion}
In this paper we studied modified Butson-Hadamard(BH) matrices  and their applications to coding theory. 
  First we obtained new results about non-ho\-mo\-ge\-ne\-ous Gray map. Then we used these results to prove some results. We give a lower bound for the minimum distance of the codes which consist  of the rows of BH matrices. Then  we determine the parameters of the codes which  occur as the images of rows of BH matrices and their translates under the homogeneous and non-homogeneous Gray maps. Here we achieve equidistant and  nonlinear codes. Also  we conclude that the codes obtained the images of the rows of BH matrices under a non-homogeneous  Gray map are Plotkin optimal codes.

\section*{Acknowledgement}

Oğuz Yayla has been supported by Middle East Technical University - Scientific Research Projects Coordination Unit under grant number 10795. Damla Acar has been supported by YÖK 100/2000 program and TÜBİTAK-BİDEB-2213-A Scholarship.

	\bibliographystyle{spmpsci}      %
	\bibliography{power}

\end{document}